\documentclass[a4paper,12pt,reqno,english]{amsart}
\usepackage[utf8]{inputenc}
\usepackage[T1]{fontenc}
\usepackage{babel}
\usepackage{float}

\usepackage[bookmarks = true, colorlinks, citecolor = blue, urlcolor = blue]{hyperref}
\usepackage{thmtools}
\usepackage[capitalize, nameinlink]{cleveref}
\usepackage[titletoc, toc, title]{appendix}

\hypersetup{pdfstartview=}

\usepackage{amssymb}
\usepackage{tikz}
\usepackage{tikz-cd}

\numberwithin{equation}{section}
\numberwithin{figure}{section}

\theoremstyle{plain}
\newtheorem{theorem}{Theorem}[section]

\theoremstyle{plain}
\newtheorem*{theorem*}{Theorem}

\theoremstyle{plain}
\newtheorem{proposition}[theorem]{Proposition}

\theoremstyle{plain}
\newtheorem{lemma}[theorem]{Lemma}

\theoremstyle{plain}
\newtheorem{corollary}[theorem]{Corollary}

\theoremstyle{definition}
\newtheorem{definition}[theorem]{Definition}

\theoremstyle{definition}

\theoremstyle{definition}

\theoremstyle{remark}
\newtheorem{remark}[theorem]{Remark}

\theoremstyle{definition}

\addto\extrasenglish{
  
}

\newcommand{\lieg}{\mathfrak{g}}

\newcommand{\Uqg}{U_q(\mathfrak{g})}
\newcommand{\Uqu}{U_q(\mathfrak{u})}

\newcommand{\UqlS}{U_q(\mathfrak{l}_S)}
\newcommand{\UqkS}{U_q(\mathfrak{k}_S)}

\newcommand{\OqG}{\mathcal{O}_q(G)}
\newcommand{\OqU}{\mathcal{O}_q(U)}

\newcommand{\OqLevi}{\mathcal{O}_q(G / L_S)}
\newcommand{\Oqflag}{\mathcal{O}_q(U / K_S)}

\newcommand{\qalg}{\mathcal{B}}

\newcommand{\id}{\mathrm{id}}

\newcommand{\bbN}{\mathbb{N}}

\newcommand{\diff}{\mathrm{d}}

\newcommand{\Ricci}{\mathrm{Ricci}}

\newcommand{\liftmap}{\ell}

\newcommand{\liftqpm}{\ell^q_{+ -}}
\newcommand{\liftqmp}{\ell^q_{- +}}

\newcommand{\diffcalc}{\Omega_q}

\newcommand{\braidS}{\hat{S}}

\newcommand{\wedgemap}{\mathord{\wedge}}

\newcommand{\catone}{{}^A_B \mathbf{Mod}_B}
\newcommand{\cattwo}{{}^A \mathbf{Mod}_B}
\newcommand{\catthree}{{}^A_B \mathbf{mod}_0}
\newcommand{\catfour}{{}^H \mathbf{mod}}
\newcommand{\catfive}{{}_{\Oqflag} \mathbf{mod}_{\Uqu}}
\newcommand{\catsix}{\mathbf{mod}_{\UqkS}}

\begin{document}

\title[The Einstein condition for quantum irreducible flag manifolds]{The Einstein condition for \\ quantum irreducible flag manifolds}

\author{Marco Matassa}

\address{Department of Computer Science, OsloMet – Oslo Metropolitan University, Oslo, Norway}

\email{marco.matassa@oslomet.no}

\begin{abstract}
We show that any quantum irreducible flag manifold satisfies an analogue of the Einstein condition, expressing proportionality between the Ricci tensor and the metric, at least in a small open interval around the classical value of the quantization parameter.
This makes use of various canonical constructions associated to these algebras, such as differential calculi and bimodule connections, which were previously introduced by various authors.
\end{abstract}

\maketitle

\section{Introduction}
\label{sec:introduction}

In differential geometry, a Riemannian manifold is said to satisfy the \emph{Einstein condition} if the Ricci tensor is proportional to the metric.
Such manifolds are the Riemannian counterparts of the solutions to Einstein's field equations for pure gravity and with cosmological constant.
For mathematical results concerning this topic we refer to \cite{besse}.

It is also possible to study an analogue of the Einstein condition in the context of \emph{non-commutative geometry}, where spaces are replaced by non-commutative algebras, which are to be considered as consisting of ``functions'' on quantum spaces.
More specifically, here we will employ the framework of \emph{quantum Riemannian geometry} as presented in \cite{quantum-book} (for an alternative approach in terms of spectral triples \cite{connes}, see for instance \cite{mesland-rennie}).
Our goal will be to establish a quantum analogue of the Einstein condition for a particular class of quantum spaces which we review below, the \emph{quantum irreducible flag manifolds}.

First of all, to any complex semisimple Lie algebra $\lieg$ we can associate the corresponding \emph{quantized enveloping algebra} $\Uqg$, as described for instance in \cite{klsc}.
Here $q$ is a real number which we refer to as the \emph{quantization parameter}. For the value $q = 1$ we recover the underlying classical object, namely the enveloping algebra $U(\lieg)$ (in a suitable sense).
Dually we can introduce the \emph{quantized coordinate ring} $\OqG$, which is a non-commutative algebra quantizing the functions on the Lie group $G$.
In this setting, quantum homogeneous spaces can be naturally defined by invariance with respect to various subalgebras of $\Uqg$ (or more generally coideals).
One notable class is given the quantum irreducible flag manifolds $\Oqflag$, which enjoy certain special properties which we describe next.

Given a non-commutative algebra $A$, we can introduce a notion of \emph{differential calculus} over $A$, which should be a differential graded algebra $\Omega^\bullet$ with various properties.
This is not canonically defined for an arbitrary algebra, in general.
However there is a canonical choice for any quantum irreducible flag manifold, which has been described by Heckenberger and Kolb in the papers \cite{locally-finite, heko}.
These \emph{Heckenberger--Kolb calculi}, denoted by $\diffcalc^\bullet$ in the following, are covariant with respect to the natural quantum group actions and are of classical graded dimension.
As such, the class of quantum irreducible flag manifolds with their differential calculi provide a natural testing ground for the development of ideas in the context of non-commutative geometry.
Even more specifically, they fit into the framework of non-commutative Kähler geometry introduced in \cite{kahler-structures}, as shown in \cite{kahler}.

Once we have a fixed a differential calculus $\Omega^\bullet$ over an algebra $A$, a \emph{quantum metric} is an element $g \in \Omega^1 \otimes_A \Omega^1$ satisfying an appropriate invertibility condition.
In the case of quantum irreducible flag manifolds, such quantum metrics were classified in \cite{levicivita-irreducible} under some natural additional conditions (such as covariance, symmetry and reality).
The outcome of this investigation is that there is a \emph{unique} (up to a scalar) quantum metric, which is then the canonical choice to investigate the Einstein condition in this setting.

Next, we can introduce connections on appropriate $A$-bimodules with respect to a fixed choice of differential calculus $\Omega^\bullet$.
In favorable cases these turn out to be \emph{bimodule connections}, which enjoy nicer properties (in the commutative setting any connection is automatically of this type).
For example, we can take the tensor product of such connections to obtain a new bimodule connection, a property that can be used to formulate a notion of \emph{compatibility} with a quantum metric.
The main result of \cite{levicivita-irreducible} is that, for any quantum irreducible flag manifold, there is a \emph{unique} connection $\nabla$ compatible with the fixed quantum metric and which is torsion-free.
This is the analogue of the \emph{Levi-Civita connection} in the classical setting, which is of central importance in Riemannian geometry.

At this point we have described most of the ingredients that are needed to formulate the Einstein condition in the non-commutative setting.
The final one is the Ricci tensor, which is defined in \cite[Section 8.1]{quantum-book} as a certain element $\Ricci_\ell \in \Omega^1 \otimes_A \Omega^1$ in terms of a given connection $\nabla$.
Here $\ell$ denotes a choice of \emph{lifting map} for the calculus, which is an extra ingredient that is necessary to make sense of this definition in the quantum setting.

With these ingredients in place, we will say that an algebra $A$ with quantum metric $g$ satisfies the \emph{Einstein condition} if $\Ricci_\ell = \lambda g$ for some scalar $\lambda$ (of course this is with respect to the fixed data of a differential calculus $\Omega^\bullet$, a connection $\nabla$ and a choice of lifting map $\ell$).
Having briefly described our setting, we can now state our main result of the paper.

\begin{theorem*}
Let $\Oqflag$ be a quantum irreducible flag manifold with the quantum metric $g$ as described above.
Then there exists a lifting map $\ell$ such that $\Oqflag$ satisfies the Einstein condition in an open interval around the classical value $q = 1$.
\end{theorem*}

This extends previous results obtained in this setting, such as the one for the quantum $2$-sphere from \cite[Section 8.2.3]{quantum-book} and its generalization to all quantum projective spaces from \cite{matassa-riemannian}.
However it should be pointed out that the cited results hold for any admissible value of $q$, rather than in just a neighborhood of the classical value $q = 1$.
Hence it remains an open problem to show whether this is the case for all quantum irreducible flag manifolds or not (this and more related problems are discussed in the last section).

The content of the paper is as follows.
In \cref{sec:preliminaries} we present in more detail the setting introduced above and recall some preliminary results we will need.
In \cref{sec:metrics-connections} we review the results from \cite{levicivita-irreducible} concerning quantum metrics and connections on quantum irreducible flag manifolds, as well as proving a result for the curvature of such connections.
In \cref{sec:lifting-maps} we construct lifting maps for the Heckenberger--Kolb calculi, satisfying various properties.
In \cref{sec:ricci-tensor} we use these results to deduce various properties of the Ricci tensor, and ultimately prove our result concerning the Einstein condition.
Finally, in \cref{sec:discussion} we offer some further discussion concerning this result as well as related open problems.

\subsection*{Acknowledgments}

We would like to thank Réamonn Ó Buachalla for his comments.

\section{Preliminaries}
\label{sec:preliminaries}

In this section we recall various notions related to quantum spaces with symmetries.

\subsection{Quantum irreducible flag manifolds}

A (generalized) \emph{flag manifold} is a homogeneous space of the form $G / P_S$, where $G$ is a complex semisimple Lie group and $P_S$ is the \emph{parabolic subgroup} corresponding to a subset $S$ of the simple roots $\Pi$.
It also admits the description $G / P_S \cong U / K_S$ as a real manifold, where $U$ is the compact real form of $G$ and $K_S = P_S \cap U$. Furthermore, we have that $K_S = L_S \cap U$, where $L_S$ is the \emph{Levi subgroup} corresponding to $S$, meaning that $K_S$ is its compact real form.
A flag manifold is called \emph{irreducible} if $S = \Pi \backslash \{\alpha_s\}$ and the simple root $\alpha_s$ has coefficient $1$ in the highest root of $\lieg$.

Let $\lieg$ be a complex semisimple Lie algebra of rank $r$.
The \emph{quantized enveloping algebra} $\Uqg$ is a Hopf algebra with generators $\{ K_i, E_i, F_i \}_{i = 1}^r$ and relations as in \cite{klsc} (for which we refer to for unexplained notions).
Here $q$ is the \emph{quantization parameter}, which we take to be a real number $q > 0$ (although note that some care is necessary to deal with the classical value $q = 1$).
We denote it by $\Uqu$ when taking into account the $*$-structure corresponding to the compact real form.
Dually we have the \emph{quantized coordinate ring} $\OqG$, which is a Hopf algebra defined in terms of the resticted dual of $\Uqg$ (more precisely, it consists of matrix coefficients of finite-dimensional type 1 representations of $\Uqg$).
As above, we write it as $\OqU$ when taking into account the corresponding $*$-structure.

The Hopf algebra $\OqG$ is naturally a $\Uqg$-bimodule by
\[
(Y \triangleright a \triangleleft Z)(X) := a(Z X Y),
\]
where $a \in \OqG$ and $X, Y, Z \in \Uqg$.
The right action $\triangleleft$ can be turned into a left action $\blacktriangleright$ by making use of the antipode, that is by setting $(Z \blacktriangleright a)(X) = a(S(Z) X)$.

\begin{remark}
We mention this (seemingly unnecessary) detail because it is the left action $\blacktriangleright$ which appears in many constructions from \cite{heko} (although this notation is not used).
It is also useful to point out that with this choice we get the module algebra relation
\[
X \blacktriangleright (a b) = (X_{(2)} \blacktriangleright a) (X_{(1)} \blacktriangleright b),
\]
which corresponds to using the opposite coproduct for $\Uqg$.
\end{remark}

Next, for any subset $S$ of simple roots we define the \emph{quantized Levi factor} by
\[
\UqlS := \langle K_i, E_j, F_j : i = 1, \ldots, r, \ j \in S \rangle \subseteq \Uqg.
\]
It is a Hopf subalgebra of $\Uqg$ for any choice of $S$.
We denote it by $\UqkS$ when considering the corresponding $*$-structure from $\Uqu$.
Finally we define
\[
\OqLevi := \{ a \in \OqG : X \triangleright a = \varepsilon(X) a, \ \forall X \in \UqlS \}.
\]
When taking into account the $*$-structure, we write this as $\Oqflag$ and refer to it as a \emph{quantum flag manifold} (see the next remark for this choice of notation).

\begin{remark}
The space $G / L_S$ can be seen as a \emph{complexification} of the flag manifold $G / P_S$, since $(G / L_S) \cap U = U / K_S$.
The points of $G / P_S$ can be recovered from $\mathcal{O}(G / L_S)$ as the \emph{real points} with respect to the involution coming from the compact real form.

Moving to the quantum setting, this means that the $*$-structure is crucial to interpret $\OqLevi$ as the quantum flag manifold $\Oqflag$.
It follows that any construction involving this algebra will have to be compatible with the $*$-structure, in order to interpret it as a construction for the quantum flag manifold.
For this reason we will insist on using the notations introduced above for the compact real forms of the relevant objects.
\end{remark}

In the following we will also adopt the shorthand notation
\[
\qalg := \Oqflag.
\]

To conclude this section, we note that all definitions so far do not require the \emph{irreducibility} condition for the quantum flag manifold $\Oqflag$.
However this will be necessary to have canonical differential calculus defined over this algebra, as we discuss next.

\subsection{Differential calculi}

For any unexplained notations in this section we refer to \cite{complex-structures}.
A \emph{differential calculus} $(\Omega^\bullet, \wedge, \diff)$ over an algebra $A$ is a differential graded algebra $\Omega^\bullet$ such that $\Omega^0 = A$ and $\Omega^\bullet$ is generated as an algebra by $A$ and $\diff A$.
If $A$ is a $*$-algebra, then we have a \emph{differential $*$-calculus} if there is an appropriate conjugate-linear involution on $\Omega^\bullet$.

A \emph{complex structure} for a differential $*$-calculus $\Omega^\bullet$ on a $*$-algebra $A$ is a $\bbN_0^2$-algebra grading $\Omega^{(\bullet, \bullet)}$, such that the following conditions are satisfied
\[
\Omega^k = \bigoplus_{p + q = k} \Omega^{(p, q)}, \quad
\left( \Omega^{(p, q)} \right)^* = \Omega^{(q, p)}, \quad
\diff \Omega^{(p, q)} \subseteq \Omega^{(p + 1, q)} \oplus \Omega^{(p, q + 1)}.
\]

We can also take into account quantum symmetries, corresponding to the case
of a Hopf algebra $H$ coacting on $A$ (say on the left). We say that $\Omega^\bullet$ is \emph{left $H$-covariant} if the left coaction extends to a comodule algebra map $\Omega^\bullet \to H \otimes \Omega^\bullet$ compatible with the differential.

It should be pointed out that, for a general non-commutative algebra $A$, there is no canonical choice of differential calculus as above.
However, in the case of quantum irreducible flag manifolds, we have the following important result due to Heckenberger and Kolb.

\begin{theorem}
There exists a unique differential $*$-calculus $(\diffcalc^\bullet, \wedge, \diff)$ over $\Oqflag$ which is left $\OqU$-covariant and having graded dimension as in the classical case.
\end{theorem}

This result is described in detail in the two papers \cite{locally-finite, heko}, except for the existence of a $*$-structure, which was introduced in \cite{kahler}.

We will refer to $\diffcalc^\bullet$ as the \emph{Heckenberger--Kolb calculus} over $\Oqflag$.
It admits a natural complex structure $\diffcalc^{(\bullet, \bullet)}$ as in \cite[Proposition 6.4]{levicivita-irreducible}, corresponding to the fact that classically these spaces are complex manifolds.
In the following we will also write
\[
\diffcalc^+ := \diffcalc^{(1, 0)}, \qquad
\diffcalc^- := \diffcalc^{(0, 1)}.
\]

\subsection{Takeuchi's equivalence}

We will now briefly recall a well-known equivalence of categories due to Takeuchi \cite{takeuchi}, which we will call \emph{Takeuchi's equivalence}. For a more detailed description of this result, and its relevance to the study of quantum homogenenous spaces, we refer to \cite[Section 2]{complex-structures} and \cite[Section 2.2]{levicivita-irreducible}.
While essential to derive many results in the cited papers, here it will only play a minor role (specifically in the proof of \cref{prop:lifting-heko}), meaning that this section can be skipped on a quick reading.

Before stating the result, we need to setup some notation.
Let $A$ and $H$ be Hopf algebras with a surjective Hopf algebra homomorphism $\pi: A \to H$.
Then we denote by $B = A^{\mathrm{co}(H)}$ the $H$-coinvariant elements of $A$ with respect to the right $H$-coaction $(\id \otimes \pi) \circ \Delta$.
In this setting, $B$ is a \emph{quantum homogeneous space} if $A$ is faithfully flat as a right $B$-module (this is known to be true for the algebras $\Oqflag$, so we will not discuss this further).

With notations as above, we define two categories as follows.
The first is the category $\catone$ of relative Hopf-modules, which here are left $A$-comodules and $B$-bimodules with a certain compatibility condition.
Next we have the category $\cattwo$ of left $H$-comodules and right $B$-modules with another compatibility condition.
There is a functor
\[
\Phi: \catone \to \cattwo
\]
which on objects is given by $\Phi(M) = M / B^+ M$ (where $B^+$ is the kernel of the counit) and on morphisms gives the corresponding maps between the quotients.
There is also another functor $\Psi$ going in the other direction, whose details we will not need in the following.

\begin{theorem}[\cite{takeuchi}]
There is a monoidal equivalence between the categories $\catone$ and $\cattwo$ given by the functors $\Phi$ and $\Psi$ (with appropriate natural transformations).
\end{theorem}

\begin{remark}
In the classical setting, this result can be translated into the fact that a $G$-equivariant vector bundle over a homogeneous space $G / H$ is completely determined by the representation of the isotropy subgroup $H$ at the identity coset.
\end{remark}

There is also a variant of this result that applies to a full subcategory of $\catone$.
This is denoted by $\catthree$ and consists objects $M$ which are finitely generated as left $B$-modules and satisfy $M B^+ = B^+ M$.
Similarly denote by $\catfour$ the category of finite-dimensional left $H$-comodules.
Then we have the following result (see \cite[Theorem 2.6]{levicivita-irreducible}).

\begin{corollary}
Takeuchi's equivalence restricts to an equivalence between $\catthree$ and $\catfour$.
\end{corollary}

In the setting of this paper, $A$ will be the quantized coordinate ring $\OqU$, $B$ the quantum irreducible flag manifold $\Oqflag$, and $H$ the Hopf algebra obtained by dualizing the injective map $\iota_S: \UqkS \hookrightarrow \Uqu$ as in \cite[Section 6.2]{levicivita-irreducible}.

Rather than working with coactions of $A$ and $H$ as above, we will dualize them to actions of their restricted duals.
Doing so will transform left $A$-comodules into (rational) right $A^\circ$-modules, see for instance \cite[Theorem 2.2.5]{hopf-book}.
In our setting, the restricted dual of $A = \OqU$ is $A^\circ = \Uqu$, while the dual of $H$ (as described above) is $H^\circ = \UqkS$.

Summarizing this discussion, we will consider the equivalence
\begin{equation}
\label{eq:categorical-equivalence}
\catfive \simeq \catsix.
\end{equation}
Here we have removed the subscript zero from the first category, to ease the notation.
Finally we note that we can turn right actions into left actions by using the antipode, which for $\OqU$ has the effect of transforming the right action $\triangleright$ of $\Uqu$ into $\blacktriangleright$, as mentioned before.

\begin{remark}
Trading coactions for actions, as above, also features at various points in the papers \cite{holomorphic-relative} and \cite{levicivita-irreducible}, however the details are not clearly spelled out there.
\end{remark}

\section{Quantum metrics and connections}
\label{sec:metrics-connections}

In this section we recall the main results from \cite{levicivita-irreducible} concerning quantum metrics and Levi-Civita connections on quantum irreducible flag manifolds.
We also prove a result concerning the curvature of such connections, which will be crucial in the following.

\subsection{Quantum metrics}

We begin by recalling the notion of quantum metric from \cite[Definition 1.15]{quantum-book}, with respect to a fixed choice of differential calculus.

\begin{definition}
Let $(\Omega^\bullet, \wedgemap, \diff)$ be a differential calculus over an algebra $A$.
Then a \emph{quantum metric} is an element $g \in \Omega^1 \otimes_A \Omega^1$ which is invertible in the following sense: there exists an $A$-bimodule map $(\cdot, \cdot) : \Omega^1 \otimes_A \Omega^1 \to A$, called the \emph{inverse metric}, such that
\[
\left( (\omega, \cdot) \otimes \id \right) (g) = \omega = \left( \id \otimes (\cdot, \omega) \right) (g)
\]
for all $\omega \in \Omega^1$. We say that $g$ is \emph{symmetric} if $\wedgemap(g) = 0$.
\end{definition}

There is also a notion of a metric being \emph{real}, for which we refer to the cited book.

In this setting, we have the following result proven in \cite[Theorem 6.12]{levicivita-irreducible}.

\begin{theorem}
\label{thm:covariant-metrics}
Let $\diffcalc^\bullet$ be the Heckenberger--Kolb calculus on a quantum irreducible flag manifold.
Then the set of covariant metrics on $\diffcalc^1 = \diffcalc^+ \oplus \diffcalc^-$ is a two-parameter family.
There is a unique (up to a scalar) quantum metric which is covariant, symmetric and real.
\end{theorem}

\begin{remark}
In \cite{matassa-fubini} there is an explicit construction of elements $g \in \diffcalc^1 \otimes_\qalg \diffcalc^1$ which are covariant, symmetric and real.
However the existence of the inverse metric $(\cdot, \cdot)$ was established only in the case of quantum projective spaces.
It follows from \cref{thm:covariant-metrics} that these must be quantum metrics for any quantum irreducible flag manifold $\qalg$.
\end{remark}

It follows from the theorem above that all quantum metrics are spanned by
\[
g_{+ -} \in \diffcalc^+ \otimes_\qalg \diffcalc^-, \qquad
g_{- +} \in \diffcalc^- \otimes_\qalg \diffcalc^+,
\]
where both elements are unique up to scalars.
In the following we will choose them in such a way that the unique (up to a scalar) \emph{symmetric} quantum metric takes the form
\[
g := g_{+ -} + g_{- +}.
\]

\subsection{Connections}

Let $(\Omega^\bullet, \wedge, \diff)$ be a differential calculus over an algebra $A$.
Then a \emph{left connection} on an $A$-bimodule $E$ is a linear map $\nabla: E \to \Omega^1 \otimes_A E$ such that the Leibniz rule $\nabla(a e) = a \nabla(e) + \diff a \otimes_A e$ holds for all $e \in E$ and $a \in A$.
It is called a \emph{bimodule connection} if there exists an $A$-bimodule map $\sigma: E \otimes_A \Omega^1 \to \Omega^1 \otimes_A E$ such that
\[
\nabla(e a) = \nabla(e) a + \sigma(e \otimes_A \diff a)
\]
for all $e \in E$ and $a \in A$.
In this case we can define a notion of compatibility between a connection $\nabla$ and a quantum metric $g$, for which we refer to \cite[Section 8.1]{quantum-book}.

Let us now consider the case of a quantum irreducible flag manifold $\qalg$ equipped with the Heckenberger--Kolb calculus $\diffcalc^\bullet$.
One of the main results of \cite{levicivita-irreducible} is the following.

\begin{theorem}[{\cite[Theorem 6.14]{levicivita-irreducible}}]
\label{thm:unique-connection}
Let $\diffcalc^\bullet$ be the Heckenberger--Kolb calculus on $\qalg$. Let $g$ be the unique (up to a scalar) quantum metric $g$ from \cref{thm:covariant-metrics}.
Then there exists a unique covariant bimodule connection $\nabla$ on $\diffcalc^1$ which is torsion-free and compatible with $g$.
\end{theorem}

Such a connection is referred to as the \emph{quantum Levi-Civita connection} (with respect to $g$).

Next, let us recall the standard notion of curvature of a connection.

\begin{definition}
Let $(\Omega^\bullet, \wedge, \diff)$ be a differential calculus over $A$.
Let $\nabla$ be a left connection on a left $A$-module $E$.
Then its \emph{curvature} is the linear map $R_\nabla: E \to \Omega^2 \otimes_A E$ defined by
\[
R_\nabla := (\diff \otimes \id - (\wedge \otimes \id) \circ (\id \otimes \nabla) \circ \nabla.
\]
\end{definition}

If $\Omega^\bullet$ carries a complex structure, in general all we can say is that the first leg of $R_\nabla$ will be an element of $\Omega^2$.
However, in our setting we can say something more specific.

The following result is a simple consequence of \cite[Theorem 5.3]{quantum-book} and the construction of the connection $\nabla$ from \cite{levicivita-irreducible}.
It is the quantum analogue of a classical property of the Levi-Civita connection on a Kähler manifold, see for instance \cite[Chapter 12]{moroianu}.

\begin{proposition}
\label{prop:curvature-chern}
Let $\nabla$ be the connection from \cref{thm:unique-connection}.
Then its curvature satisfies
\[
R_{\nabla}(\diffcalc^{\pm}) \subseteq \diffcalc^{(1, 1)} \otimes_{\qalg} \diffcalc^{\pm}.
\]
\end{proposition}

\begin{proof}
In \cite[Theorem 6.14]{levicivita-irreducible} the connection $\nabla$ is constructed as $\nabla = \nabla_{\mathrm{Ch}} + \nabla_{\mathrm{Ch, op}}$.
Here $\nabla_{\mathrm{Ch}}$ is the Chern connection on $\diffcalc^+$ (see the cited reference for this definition) and $\nabla_{\mathrm{Ch, op}}$ is the Chern connection on $\diffcalc^-$ with respect to the opposite complex structure.

By \cite[Theorem 5.3]{quantum-book}, the Chern connection obeys the following important property: the components $\diffcalc^{(2, 0)}$ and $\diffcalc^{(0, 2)}$ of its curvature both vanish.
(Note that this result requires the property of factorizability of the calculus, which indeed holds for $\diffcalc^\bullet$).
Hence we get
\[
R_{\nabla_{\mathrm{Ch}}}(\diffcalc^+) \subseteq \diffcalc^{(1, 1)} \otimes \diffcalc^+.
\]
The argument is analogous for the connection $\nabla_{\mathrm{Ch, op}}$, which is defined by considering the differential calculus with its opposite complex structure.
Concretely this means that the graded component $\diffcalc^{(a, b)}$ is exchanged with $\diffcalc^{(b, a)}$, which leaves $\diffcalc^{(1, 1)}$ invariant.
\end{proof}

Because of this result, in the following we will only need a lifting map $\diffcalc^2 \to \diffcalc^1 \otimes_\qalg \diffcalc^1$ to be defined on the subspace $\diffcalc^{(1, 1)}$.
This notion is recalled in the next section.

\section{Construction of lifting maps}
\label{sec:lifting-maps}

As mentioned in the introduction, a lifting map is an extra piece of data necessary in the quantum setting to define the Ricci tensor.
We will construct such maps for the Heckenberger--Kolb calculi, which we will ultimately use to prove the Einstein condition.

\subsection{Lifting maps}

We now recall the concept of lifting map for a differential calculus, which will be used later on to define the Ricci tensor in the quantum setting.

\begin{definition}
Let $(\Omega^\bullet, \wedgemap, \diff)$ be a differential calculus over an algebra $A$.
Then a \emph{lifting map} is an $A$-bimodule map $\liftmap: \Omega^2 \to \Omega^1 \otimes_A \Omega^1$ such that $\wedgemap \circ \liftmap = \id$.
\end{definition}

We have the following elementary observation.

\begin{lemma}
\label{lem:convex-lifting}
Suppose $\liftmap_1$ and $\liftmap_2$ are lifting maps.
Then $\liftmap = c_1 \liftmap_1 + c_2 \liftmap_2$ is a lifting map if and only if $c_1 + c_2 = 1$, in other words if $\liftmap$ is a convex combination of the two maps.
\end{lemma}

\begin{proof}
Clearly $\liftmap = c_1 \liftmap_1 + c_2 \liftmap_2$ is an $A$-bimodule map. We compute
\[
\wedgemap \circ \liftmap = \wedgemap \circ (c_1 \liftmap_1 + c_2 \liftmap_2) = (c_1 + c_2) \id.
\]
Hence $\wedgemap \circ \liftmap = \id$ if and only if $c_1 + c_2 = 1$.
\end{proof}

Our goal in the next subsection will be to prove the existence of certain (equivariant) lifting maps for the Heckenberger--Kolb calculus $\diffcalc^\bullet$ over a quantum irreducible flag manifold.

\subsection{Heckenberger--Kolb calculus}

In the following we will write
\[
V^{(a, b)} := \Phi(\diffcalc^{(a, b)}),
\]
where $\Phi$ is Takeuchi's functor. More concretely, we have $V^{(a, b)} = \diffcalc^{(a, b)} / \qalg^+ \diffcalc^{(a, b)}$.
All these (finite-dimensional) vector spaces are modules for the quantized Levi factor $\UqkS$.

We will need the following concrete description of the component $V^{(1, 1)}$.

\begin{lemma}[\cite{heko}]
\label{lem:characterization-mixed}
We have that $V^{(1, 1)}$ is isomorphic to the quotient of
\[
(V^{(1, 0)} \otimes V^{(0, 1)}) \oplus (V^{(0, 1)} \otimes V^{(1, 0)})
\]
by the subspace of relations of the form
\[
\{ v \otimes w + \braidS(v \otimes w) : v \in V^{(1, 0)}, \ w \in V^{(0, 1)} \}.
\]
Here $\braidS: V^{(1, 0)} \otimes V^{(0, 1)} \to V^{(0, 1)} \otimes V^{(1, 0)}$ is a certain invertible $\UqkS$-equivariant map, which reduces to the flip map in the classical limit $q = 1$.
\end{lemma}

\begin{proof}
The first statement is the content of \cite[Proposition 3.11(ii)]{heko}.
That $\braidS$ is an invertible $\UqkS$-equivariant map equivariant map follows from \cite[Lemma 2.2]{heko}, where it is shown that $\braidS$ is a multiple of the braiding $\hat{R}_{V^{(1, 0)}, V^{(0, 1)}}$ of $\UqkS$-modules.
That $\braidS$ reduces to the flip map follows from the corresponding result for the braiding $\hat{R}$, together with the fact that the two coincide up to an appropriate power of $q$.
\end{proof}

We are now ready to prove the following existence result for lifting maps.

\begin{proposition}
\label{prop:lifting-heko}
Let $\diffcalc^\bullet$ be the Heckenberger--Kolb calculus over $\qalg$.
Then there exist two $\Uqu$-equivariant lifting maps $\liftqpm$ and $\liftqmp$ such that
\[
\liftqpm : \diffcalc^{(1, 1)} \to \diffcalc^+ \otimes_\qalg \diffcalc^-
\qquad \textrm{and} \qquad
\liftqmp : \diffcalc^{(1, 1)} \to \diffcalc^- \otimes_\qalg \diffcalc^+.
\]
In the classical limit $q = 1$ they reduce to the maps
\[
x \wedge y \mapsto x \otimes y,
\quad \textrm{and} \quad
x \wedge y \mapsto - y \otimes x.
\]
\end{proposition}

\begin{proof}
By Takeuchi's equivalence, as in \eqref{eq:categorical-equivalence}, it suffices to prove the analogous result for the corresponding $\UqkS$-modules.
That is, we are looking for $\UqkS$-equivariant maps
\[
s_{+ -}: V^{(1, 1)} \to V^{(1, 0)} \otimes V^{(0, 1)}
\quad \textrm{and} \quad
s_{- +}: V^{(1, 1)} \to V^{(0, 1)} \otimes V^{(1, 0)},
\]
such that composing them with the quotient map $\wedgemap$ gives the identity.

Write $V := (V^{(1, 0)} \otimes V^{(0, 1)}) \oplus (V^{(0, 1)} \otimes V^{(1, 0)})$ and note that $V^{(1, 1)}$ is a quotient of $V$ by \cref{lem:characterization-mixed}.
Consider the map $\tilde{s}_{+ -} : V \to V^{(1, 0)} \otimes V^{(0, 1)}$ defined by
\[
\tilde{s}_{+ -}(v \otimes w) = v \otimes w, \quad
\tilde{s}_{+ -}(w \otimes v) = - \braidS^{-1}(w \otimes v).
\]
Here $v \in V^{(1, 0)}$ and $w \in V^{(0, 1)}$, while $\braidS^{-1}$ is the inverse of the map from \cref{lem:characterization-mixed}.
We have
\[
\tilde{s}_{+ -} \left( v \otimes w + \braidS(v \otimes w) \right) = v \otimes w - \braidS^{-1} \circ \braidS(v \otimes w) = 0.
\]
Then $\tilde{s}_{+ -}$ descends to the quotient $V^{(1, 1)}$ by \cref{lem:characterization-mixed}, so that we obtain a well-defined map $s_{+ -} : V^{(1, 1)} \to V^{(1, 0)} \otimes V^{(0, 1)}$.
This map is $\UqkS$-equivariant because $\braidS$ is $\UqkS$-equivariant.
It is also clear that $\wedgemap \circ s_{+ -} = \id$ by construction.

Similarly, we obtain a $\UqkS$-equivariant map $s_{- +} : V^{(1, 1)} \to V^{(0, 1)} \otimes V^{(1, 0)}$ starting from
\[
\tilde{s}_{- +}(v \otimes w) = - \braidS(v \otimes w), \quad
\tilde{s}_{- +}(w \otimes v) = w \otimes v,
\]
and applying similar considerations as those given above.

Finally let us discuss what happens in the classical limit $q = 1$.
We know that $\braidS$ reduces to the flip map by \cref{lem:characterization-mixed}.
It follows that $s_{+ -}$ and $s_{- +}$ reduce to the maps $v \wedge w \mapsto v \otimes w$ and $v \wedge w \mapsto - w \otimes v$, respectively.
This observation, together with Takeuchi's equivalence, establishes the conclusion for the classical limits of the lifting maps $\liftqpm$ and $\liftqmp$.
\end{proof}

\begin{remark}
These are not not necessarily all $\Uqu$-equivariant lifting maps for $\diffcalc^2$. For instance, it is possible to modify the construction above by assigning different signs to the various simple components of the relevant tensor products.

The focus on the lifting maps defined in \cref{prop:lifting-heko} is because of their classical limits, a property which will ultimately used in the proof of \cref{thm:einstein-condition}.
\end{remark}

\section{Ricci tensor and Einstein condition}
\label{sec:ricci-tensor}

This section contains the main results of the paper, which concern the Einstein condition for quantum irreducible flag manifolds.
After recalling the definition of the Ricci tensor in this context, we prove that it takes a special form provided we use appropriate lifting maps.
Once this is done, we prove a first result that hints at the fact that the Einstein condition should hold for generic $q$.
Finally, we use the classical limit to deduce that the Einstein condition holds in a open interval around the classical value $q = 1$.

\subsection{Definitions}

We begin by recalling the definition of the Ricci tensor in the context of quantum Riemannian geometry, following \cite[Section 8.1]{quantum-book}.

\begin{definition}
Let $(\Omega^\bullet, \wedgemap, \diff)$ be a differential calculus over an algebra $A$. Let $g$ be a quantum metric with inverse metric $(\cdot, \cdot)$. Let $\nabla$ be a left connection on $\Omega^1$ and $\ell: \Omega^2 \to \Omega^1 \otimes_A \Omega^1$ be a lifting map.
Then the \emph{Ricci tensor} (for the map $\ell$ as above) is defined by
\[
\Ricci_\ell := ((\cdot, \cdot) \otimes \id \otimes \id) \circ (\id \otimes \ell \otimes \id) \circ (\id \otimes R_\nabla)(g) \in \Omega^1 \otimes_A \Omega^1.
\]
\end{definition}

\begin{remark}
The subscript here is meant to stress that this element depends on the choice of the lifting map $\ell$.
This notation will also make some later arguments more transparent.
\end{remark}

Suppose $A$ consists of functions on a classical space, with $\Omega^1$ the usual space of $1$-forms.
Then $\Ricci_\ell$ reduces to the ordinary Ricci tensor provided we use the lifting map
\[
\ell(x \wedge y) = \frac{1}{2} (x \otimes y - y \otimes x), \quad
x, y \in \Omega^1.
\]
For details regarding this claim in the classical setting see \cite[Example 8.10]{quantum-book}.

We can then naturally define a quantum analogue of the Einstein condition.

\begin{definition}
With notations as above, we say that an algebra $A$ with a fixed quantum metric $g$ (and additional data specified before) satisfies the \emph{Einstein condition} if we have
\[
\Ricci_\ell = \lambda g.
\]
Here $\lambda$ is a scalar which we call the \emph{Einstein constant}.
\end{definition}

Note that this condition depends crucially on the choice of the lifting map $\ell$, unlike in the classical case (where it is always with respect to the antisymmetrization map as above).

Our goal in the following will be to check whether this condition holds for the case of quantum irreducible flag manifolds (with quantum metrics as previously discussed).

\subsection{Results for the Ricci tensor}

Let $\qalg = \Oqflag$ be a quantum irreducible flag manifold with the corresponding Heckenberger--Kolb calculus $\diffcalc^\bullet$.
We consider the unique (up to a scalar) quantum metric from \cref{thm:covariant-metrics} and the corresponding unique connection from \cref{thm:unique-connection}.
We will investigate $\Ricci_\liftmap$ for various choices of lifting maps.

We should note that, in all arguments that will follow, we will only need a lifting map $\liftmap$ to be defined on the subspace $\diffcalc^{(1, 1)} \subset \diffcalc^2$.
This is a consequence of \cref{prop:curvature-chern} concerning the curvature of the connection $\nabla$ (also see the next proof for details).

\begin{lemma}
\label{lem:ricci-lifting}
Let $\Ricci_\liftmap$ be the Ricci tensor for the quantum flag manifold $\qalg$, defined with respect to a lifting map $\liftmap$.
We have the following properties:
\begin{enumerate}
\item if $\liftmap : \diffcalc^{(1, 1)} \to \diffcalc^+ \otimes_\qalg \diffcalc^-$ then $\Ricci_\liftmap \in \diffcalc^- \otimes_\qalg \diffcalc^+$,
\item if $\liftmap : \diffcalc^{(1, 1)} \to \diffcalc^- \otimes_\qalg \diffcalc^+$ then $\Ricci_\liftmap \in \diffcalc^+ \otimes_\qalg \diffcalc^-$.
\end{enumerate}
\end{lemma}

\begin{proof}
(1) Using \cref{prop:curvature-chern} we have
\[
\begin{split}
(\id \otimes R_\nabla)(g_{+ -}) & \in \diffcalc^+ \otimes_\qalg \diffcalc^{(1, 1)} \otimes_\qalg \diffcalc^-, \\
(\id \otimes R_\nabla)(g_{- +}) & \in \diffcalc^- \otimes_\qalg \diffcalc^{(1, 1)} \otimes_\qalg \diffcalc^+.
\end{split}
\]
Suppose $\liftmap : \diffcalc^{(1, 1)} \to \diffcalc^+ \otimes_\qalg \diffcalc^-$. Then applying this map gives
\[
\begin{split}
(\id \otimes \liftmap \otimes \id) \circ (\id \otimes R_\nabla)(g_{+ -}) & \in \diffcalc^+ \otimes_\qalg \diffcalc^+ \otimes_\qalg \diffcalc^- \otimes_\qalg \diffcalc^-, \\
(\id \otimes \liftmap \otimes \id) \circ (\id \otimes R_\nabla)(g_{- +}) & \in \diffcalc^- \otimes_\qalg \diffcalc^+ \otimes_\qalg \diffcalc^- \otimes_\qalg \diffcalc^+.
\end{split}
\]
It follows from \cite[Theorem 3.6]{levicivita-irreducible} that the inverse metric $(\cdot, \cdot)$ is zero on $\diffcalc^+ \otimes \diffcalc^+$ (and similarly it is zero on $\diffcalc^- \otimes \diffcalc^-$).
Using this and $g = g_{+ -} + g_{- +}$ we get
\[
\Ricci_\liftmap = ((\cdot, \cdot) \otimes \id \otimes \id) \circ (\id \otimes \liftmap \otimes \id) \circ (\id \otimes R_\nabla)(g_{- +}) \in \diffcalc^- \otimes_\qalg \diffcalc^+.
\]

(2) The argument for the case $\liftmap : \diffcalc^{(1, 1)} \to \diffcalc^- \otimes \diffcalc^+$ is completely analogous.
\end{proof}

We can now use the properties above, together with the uniqueness (up to scalars) of the quantum metrics from \cref{thm:covariant-metrics}, to derive the following result.

\begin{proposition}
\label{prop:ricci-proportional}
If $\liftmap$ is a $\Uqu$-equivariant lifting map then $\Ricci_\liftmap$ is a $\Uqu$-invariant element of $\diffcalc^1 \otimes_\qalg \diffcalc^1$. Furthermore, we have the following properties:
\begin{enumerate}
\item if $\liftmap_{+ -} : \diffcalc^{(1, 1)} \to \diffcalc^+ \otimes_\qalg \diffcalc^-$ then $\Ricci_{\liftmap_{+ -}} = a g_{+ -}$ for some coefficient $a$,
\item if $\liftmap_{- +} : \diffcalc^{(1, 1)} \to \diffcalc^- \otimes_\qalg \diffcalc^+$ then $\Ricci_{\liftmap_{- +}} = b g_{- +}$ for some coefficient $b$.
\end{enumerate}
\end{proposition}

\begin{proof}
If $\liftmap$ is $\Uqu$-equivariant, then all maps appearing in the definition of the Ricci tensor are $\Uqu$-equivariant.
Since the composition of these maps is applied to the $\Uqu$-invariant element $g$, it follows that the element $\Ricci_\liftmap$ must be $\Uqu$-invariant.

(1) Here we are now assuming that $\liftmap_{+ -} : \diffcalc^{(1, 1)} \to \diffcalc^+ \otimes_\qalg \diffcalc^-$, together with $\Uqu$-equivariance as above.
Then \cref{lem:ricci-lifting} shows that $\Ricci_{\liftmap_{+ -}}$ must be a $\Uqu$-invariant element of $\diffcalc^- \otimes \diffcalc^+$.
It follows that $\Ricci_{\liftmap_{+ -}}$ must be a multiple of $g_{+ -}$ by \cref{thm:covariant-metrics}.

(2) The proof in the case of $\liftmap_{- +} : \diffcalc^{(1, 1)} \to \diffcalc^- \otimes_\qalg \diffcalc^+$ is completely analogous.
\end{proof}

\begin{remark}
In the proof above, we do not need the full result of \cref{thm:covariant-metrics}. The only fact which is needed is that all $\Uqu$-invariant elements of $\diffcalc^1 \otimes_\qalg \diffcalc^1$ are spanned by unique (up to scalars) elements in $\diffcalc^+ \otimes_\qalg \diffcalc^-$ and $\diffcalc^- \otimes_\qalg \diffcalc^+$.
This easily follows from Takeuchi's equivalence, since $V^1 = \Phi(\diffcalc^1)$ decomposes as $V^1 = V^{(1, 0)} \oplus V^{(0, 1)}$, where both summands are non-isomorphic simple $\UqkS$-modules (see also \cite[Lemma 6.8]{levicivita-irreducible}).
\end{remark}

\subsection{The Einstein condition}

Before discussing the Einstein condition for quantum irreducible flag manifolds, it is interesting to note that in our setting it is equivalent to the requirement that the Ricci tensor should be \emph{symmetric}. The precise result is as follows.

\begin{proposition}
Let $\qalg$ be a quantum irreducible flag manifold, with Heckenberger--Kolb calculus $\diffcalc^\bullet$ and quantum metric $g$.
Let $\liftmap$ be a $\Uqu$-equivariant lifting map for $\diffcalc^2$.
Then the Einstein condition holds if and only if the Ricci tensor is symmetric, that is $\wedgemap(\Ricci_\liftmap) = 0$.
\end{proposition}

\begin{proof}
Given a $\Uqu$-equivariant lifting map $\liftmap$, \cref{prop:ricci-proportional} shows that
\[
\Ricci_\liftmap = t_1 g_{+ -} + t_2 g_{- +},
\]
for some scalars $t_1$ and $t_2$.
Since $g = g_{+ -} + g_{- +}$, 
the Einstein condition $\Ricci_\liftmap = \lambda g$ holds if and only if $t_1 = t_2$ (with $\lambda$ being equal to this common value).

On the other hand, recall that the quantum metric $g = g_{+ -} + g_{- +}$ is symmetric, that is $\wedgemap(g) = 0$.
This implies that $\wedgemap(g_{- +}) = - \wedgemap(g_{+ -})$, which allows us to write
\[
\wedgemap(\Ricci_\liftmap) = (t_1 - t_2) \wedgemap(g_{+ -}).
\]
It is not difficult to show that $\wedgemap(g_{+ -}) \neq 0$ (either by using invertibility of the quantum metric or by computing its image in the quotient $V^{(1, 1)}$).
Then $\wedgemap(\Ricci_\liftmap) = 0$ if and only if $t_1 = t_2$.
This establishes the equivalence of the two conditions from the claim.
\end{proof}

\begin{remark}
This result is specific to quantum irreducible flag manifolds, meaning that we should not expect such an equivalence to be true for more general algebras.
\end{remark}

We are now in the position to show that the Einstein condition holds for $\qalg$, provided we have a certain condition on the coefficients $a$ and $b$ from \cref{prop:ricci-proportional}.

\begin{proposition}
\label{prop:einstein-condition}
Let $\liftmap_{+ -}$ and $\liftmap_{- +}$ be two $\Uqu$-equivariant lifting maps as in \cref{prop:ricci-proportional}.
Denote by $a$ and $b$ the corresponding coefficients as defined there.

Suppose it holds that $a + b \neq 0$.
Then there exists a unique lifting map $\liftmap = c_1 \liftmap_{+ -} + c_2 \liftmap_{- +}$ such that the Einstein condition holds with non-zero Einstein constant, that is
\[
\Ricci_\liftmap = \lambda g, \quad \lambda \neq 0.
\]
In this case, the coefficients $c_1$ and $c_2$ for the lifting map $\liftmap$ are given by
\[
c_1 = \frac{b}{a + b}, \quad
c_2 = \frac{a}{a + b}.
\]
\end{proposition}

\begin{proof}
Note that, by \cref{lem:convex-lifting}, we have that $\liftmap = c_1 \liftmap_{+ -} + c_2 \liftmap_{- +}$ is a lifting map if and only if it is a convex combination of the two lifting maps $\liftmap_{+ -}$ and $\liftmap_{- +}$, which implies that $c_1 + c_2 = 1$.
It is clearly $\Uqu$-equivariant, since this is the case for both maps considered.

It immediately follows from the definition of the Ricci tensor that
\[
\Ricci_\ell = c_1 \Ricci_{\ell_{+ -}} + c_2 \Ricci_{\ell_{- +}}.
\]
Then, by \cref{prop:ricci-proportional}, there exist coefficients $a$ and $b$ such that
\[
\Ricci_\ell = c_1 a g_{+ -} + c_2 b g_{- +}.
\]
Hence this expression is proportional to the quantum metric $g = g_{+ -} + g_{- +}$ if and only if $c_1 a = c_2 b$.
Since $c_1 + c_2 = 1$, this is possible if and only if
\[
c_1 = \frac{b}{a + b}, \quad
c_2 = \frac{a}{a + b},
\]
which clearly requires that $a + b \neq 0$.
In this case, the Einstein condition holds with Einstein constant $\lambda = c_1 a = c_2 b$, which is non-zero by the assumption $a + b \neq 0$.
\end{proof}

\begin{remark}
Observe that in the result above we have only used some general properties of the lifting maps $\liftmap_{+ -}$ and $\liftmap_{- +}$, but not their specific form.
\end{remark}

Note that lifting maps with the properties required in \cref{prop:einstein-condition} do exist, as we have shown in \cref{prop:lifting-heko}.
This means that we can expect the Einstein condition to hold for \emph{generic} quantization parameter $q$, since the condition $a + b \neq 0$ should hold generically.

Below we will show that this is indeed the case, at least provided $q$ is sufficiently close to the classical value $q = 1$.
Here we will use the lifting maps $\liftqpm$ and $\liftqmp$ constructed in \cref{prop:lifting-heko}, which allow us to make considerations based on the classical limit.

\begin{theorem}
\label{thm:einstein-condition}
Let $\Oqflag$ be a quantum irreducible flag manifold.
Let $\diffcalc^\bullet$ be the corresponding Heckenberger--Kolb calculus, $g$ the previously introduced quantum metric, and $\nabla$ the unique Levi-Civita connection compatible with this quantum metric.

Then there is a unique lifting map $\liftmap^q = c_1(q) \liftqpm + c_2(q) \liftqmp$ such that the Einstein condition holds, with non-zero Einstein constant, in an open interval around the classical value $q = 1$.
\end{theorem}

\begin{proof}
Let us write $a(q)$ and $b(q)$ for the coefficients $a$ and $b$ from \cref{prop:einstein-condition}, to stress their dependence on the quantization parameter $q$. More explicitly, we have
\begin{equation}
\label{eq:ricci-q}
\Ricci_{\liftqpm} = a(q) g_{+ -}, \quad
\Ricci_{\liftqmp} = b(q) g_{- +}.
\end{equation}
Then \cref{prop:einstein-condition} shows that the Einstein condition holds provided we have $a(q) + b(q) \neq 0$, and this determines uniquely the scalars $c_1(q)$ and $c_2(q)$.
Hence it suffices to show that the condition $a(q) + b(q) \neq 0$ is satisfied in an appropriate open interval around $q = 1$.

In turn, this will follow if we can prove that:
\begin{enumerate}
\item $a(q)$ and $b(q)$ are continuous functions in an open interval around $q = 1$,
\item at the classical value $q = 1$ we have that $a(1) + b(1) \neq 0$.
\end{enumerate}

For (1), we claim that all constructions considered here are continuous in a small interval around a fixed value $q_0$ of the quantization parameter.
This is well-known for the quantized coordinate ring $\OqU$ and similarly for $\Oqflag$. Continuity of the Heckenberger-Kolb calculus $\diffcalc^\bullet$ is discussed in \cite[Lemma 5.7]{kahler}.
Similarly this can be observed for all constructions appearing in \cite{levicivita-irreducible}. Indeed, they follow from categorical considerations which use morphisms that do satisfy this property (in turn, this follows from well-known results on $\Uqu$-modules).
This guarantees that (1) holds for $a(q)$ and $b(q)$.

As for (2), we need to consider the classical limit $q = 1$.
Here $\Oqflag$ reduces to a suitable algebra of functions $A$ on the irreducible flag manifold $G / P_S \cong U / K_S$.
Similarly, $\diffcalc^\bullet$ reduces to the ordinary forms over this space (in the sense of Kähler differentials), which we denote by $\Omega^\bullet$.
The quantum metric $g = g_{+ -} + g_{- +}$ reduces to a $\lieg$-invariant metric on $G / P_S$, while $\nabla$ becomes the (unique) classical Levi-Civita connection with respect to this metric.
Note that for $g$ and $\nabla$ we suppress any dependence on $q$, for notational simplicity.

Within this setting, consider the classical lifting map
\[
\ell: \Omega^2 \to \Omega^1 \otimes_A \Omega^1, \quad
\ell(x \wedge y) = \frac{1}{2} (x \otimes y - y \otimes x).
\]
Then \cite[Example 8.10]{quantum-book} shows that $\Ricci_\ell$ coincides with the usual Ricci tensor (up to conventional choices).
Since the irreducible flag manifold $G / P_S$ is an Einstein manifold, this means that $\Ricci_\ell = \lambda g$, where $\lambda$ is the classical Einstein constant.

Note that the classical lifting map $\ell$, as described above, can be considered as the convex combination $\frac{1}{2} (\liftmap_{+ -} + \liftmap_{+ -})$ of the two $\lieg$-equivariant lifting maps
\begin{align*}
\ell_{+ -} & : \Omega^{(1, 1)} \to \Omega^+ \otimes_A \Omega^-, & \ell_{+ -}(x \wedge y) & = x \otimes y, \\
\ell_{- +} & : \Omega^{(1, 1)} \to \Omega^- \otimes_A \Omega^+, & \ell_{- +}(x \wedge y) & = - y \otimes x.
\end{align*}
Here we write $x \in \Omega^+$ and $y \in \Omega^-$.
Upon comparison with \cref{prop:lifting-heko}, we see that these coincide with the classical limits $\liftmap^1_{+ -}$ and $\liftmap^1_{- +}$ of the maps $\liftqpm$ and $\liftqmp$ defined there.

Hence we have the identity $\liftmap = \frac{1}{2} (\liftmap^1_{+ -} + \liftmap^1_{+ -})$. Using this and \eqref{eq:ricci-q} we get
\[
\Ricci_{\ell} = \frac{1}{2} \Ricci_{\ell^1_{+ -}} + \frac{1}{2} \Ricci_{\ell^1_{- +}} = \frac{1}{2} \left( a(1) g_{+ -} + b(1) g_{- +} \right).
\]
Since we know that $\Ricci_\ell = \lambda g$, this implies that $a(1) = b(1) = 2 \lambda$. In particular, we find that both $a(1)$ and $b(1)$ are non-zero, since the Einstein constant $\lambda$ is non-zero for any irreducible flag manifold $G / P_S$ (see for instance \cite[Chapter 8]{besse}, which discusses the more general class of compact homogeneous Kähler manifolds).
This proves the result.
\end{proof}

\begin{remark}
As an extra observation, it follows from the above that
\[
c_1(1) = \frac{b(1)}{a(1) + b(1)} = \frac{1}{2}, \quad
c_2(1) = \frac{a(1)}{a(1) + b(1)} = \frac{1}{2},
\]
which implies that $\liftmap^q = c_1(q) \liftqpm + c_2(q) \liftqmp$ reduces to the classical map at $q = 1$.
\end{remark}

\section{Discussion and outlook}
\label{sec:discussion}

The main result of this paper is \cref{thm:einstein-condition}, which proves the Einstein condition for any quantum irreducible flag manifold $\Oqflag$ in an open interval around the classical value $q = 1$.
In this last section we discuss some related open questions and generalizations.

The first natural question is whether the Einstein condition holds for all $q > 0$. It should be possible, by a more careful analysis of the results, to prove that the result should hold for all but possibly finitely many values of $q$.
This is because most construction take place over the field of rational functions in $q$ (although note that dealing with $*$-structures can lead to the introduction of square-roots, for instance to obtain orthonormal bases).

In \cite[Theorem 6.5]{matassa-riemannian} it was shown that the Einstein condition holds for all quantum projective spaces, generalizing the previous result for the quantum $2$-sphere discussed in \cite[Section 8.2.3]{quantum-book}.
Notably these results holds for \emph{all values of $q$}, which makes it plausible that this should be the case for any quantum irreducible flag manifold as well.

Another natural question is whether it is possible to extend these results to quantum flag manifolds which are \emph{not irreducible}.
Asking for this generalization makes sense, since any (generalized) flag manifold is an Einstein manifold (as they all are compact homogeneous Kähler manifolds).
But this might not be possible, and in general the problem appears much harder than what we discussed here.
In \cite{equivariant} it is shown that it is possible to introduce covariant \emph{first-order} calculi generalizing those of Heckenberger--Kolb.
However it is a highly non-trival task to produce a reasonable analogue of the full calculus, as discussed in \cite{full-flag} for the case of the quantum full flag manifold corresponding to $SU(3)$.
For example, in the cited paper it is shown that it is not possible to have a covariant quantum metric in the sense considered here, which implies that the whole setting must be modified significantly.

\end{document}